\documentclass[12pt]{article}
\usepackage{amsmath,amssymb,amsthm}
\usepackage[colorlinks=true,linkcolor=blue,citecolor=blue,urlcolor=blue]{hyperref}

\title{An Inductive Proof that Lights Out Configurations are Invertible, and a Parity-Invariance Result}
\author{Keivan Mirzaei\thanks{Email: keivan.mirzaei@ucalgary.ca}}
\date{\today}

\newtheorem{theorem}{Theorem}
\newtheorem{proposition}[theorem]{Proposition}
\newtheorem{remark}[theorem]{Remark}

\begin{document}
\maketitle

\begin{abstract}
We give an elementary inductive proof of a classical result for the \emph{Lights Out problem} on graphs: from any configuration of vertices, one can reach the complementary configuration by a sequence of moves, where a move consists of toggling a vertex and its neighbors. We also prove, again by a purely elementary argument, a parity-invariance property: once an initial configuration is fixed, the parity of the number of presses required to reach an attainable configuration is determined by that configuration. In particular, any two solutions leading to the same attainable configuration differ by an even number of presses.
\end{abstract}

\section{Introduction}
The \emph{Lights Out problem} has been studied from several points of view, most commonly via linear algebra over $\mathbb{F}_2$; see, for example, Anderson and Feil \cite{AndersonFeil1998}, the survey of Fleischer and Yu \cite{FleischerYu2013}, the expository account of Kreh \cite{Kreh2017}, and the graph-theoretic treatment of Berman, Borer, and Hungerb\"uhler \cite{BermanBorerHungerbuhler2021}. In the graph-theoretic language, the press sets solving the all-on configuration correspond to odd dominating sets, and this connects the game with the broader literature on parity domination; see again \cite{BermanBorerHungerbuhler2021,FleischerYu2013}.

In contrast, the present note emphasizes elementary arguments. Our first goal is to give a purely inductive and combinatorial proof that every configuration can be complemented. Our second goal is to record a simple parity-invariance phenomenon for solutions: for a fixed initial configuration, the parity of the number of presses depends only on the attained configuration. Related parity phenomena appear in the literature, typically in the language of quiet patterns or neutral procedures. For instance, Meyer proves that every quiet pattern in a symmetric reflexive Lights Out game has even cardinality \cite{Meyer2013}, and Mart\'\i n-S\'anchez and Pareja-Flores obtain related parity restrictions via neutral procedures \cite{MartinSanchezParejaFlores2001}. The result below is equivalent in substance to these observations, but appears to be new in this direct formulation. In addition, unlike the linear-algebraic arguments in those works over $\mathbb{F}_2$, our proof is completely elementary and uses only cancellation of repeated presses together with the handshaking lemma.

\section{Main Result}

\begin{theorem}
Let $G$ be a simple graph with $n$ vertices. From any initial configuration of the vertices, it is possible to reach the complementary configuration by a sequence of operations, where each operation toggles a vertex and all of its neighbors.
\end{theorem}

\begin{proof}
We proceed by induction on $n$. For $n=1$ the claim is immediate. Assume it holds for graphs on $k$ vertices, and let $G$ have $k+1$ vertices $V=\{v_1,\dots,v_{k+1}\}$.

For each $i$, consider the induced subgraph $G_i$ on $V\setminus\{v_i\}$. By the induction hypothesis, there is a sequence $P_i$ of operations supported on $G_i$ that toggles all vertices of $G_i$. If for some $i$ this sequence also toggles $v_i$, then $P_i$ already toggles every vertex of $G$, and we are done. Thus we may assume that, for the indices under consideration, $P_i$ leaves $v_i$ unchanged.

Now take two distinct vertices $v_i$ and $v_j$. Applying $P_i$ followed by $P_j$ toggles each vertex of $V\setminus\{v_i,v_j\}$ twice (so they remain unchanged), while toggling $v_i$ and $v_j$ once. Hence any pair of vertices can be toggled simultaneously, leaving all other vertices unaffected.

To obtain the complementary configuration, we first use the previous step to toggle all the vertices of odd degree (if any) in pairs, which is possible because the number of odd-degree vertices is even. Finally, we toggle every vertex once. The vertices of odd degree have then been toggled twice in total, while the vertices of even degree have been toggled once, so every vertex changes state exactly once. Thus the complementary configuration is reached. By induction, the result holds for all $n$.
\end{proof}

\begin{remark}
The standard proof of the theorem above uses the adjacency-plus-identity matrix over $\mathbb{F}_2$. The argument given here is instead purely inductive and combinatorial.
\end{remark}

\section{A Parity-Invariance Result}

We shall use two basic observations. First, the order of presses does not matter. Second, pressing the same vertex twice has no net effect. Consequently, every press sequence may be identified with the subset of vertices pressed an odd number of times.

\begin{proposition}[Parity invariance]
For any attainable configuration in
the Lights Out problem, the parity of the number of button
presses is uniquely determined by the configuration. In particular,
any two solutions differ by an even number of presses.
\end{proposition}

\begin{proof}
It is enough to show that the initial configuration has even parity. In fact, if a configuration were attainable by both an even and an odd number of moves, then performing these two sequences consecutively would return the system to the initial configuration using an odd number of presses.

Now let $S\subseteq V(G)$ be a press set that returns the system to its initial configuration. Consider the induced subgraph $G[S]$. For a vertex $v\in S$, $v$ is toggled by its own press and all of its neighbors in $S$. So in order to end unchanged it must have an odd degree in $G[S]$. Therefore every vertex of $G[S]$ has odd degree and hence $G[S]$ has an even number of vertices according to the handshaking lemma.
\end{proof}

\begin{remark}
Although equivalent parity phenomena appear in the literature in the context of quiet patterns and neutral procedures; see \cite{Meyer2013,MartinSanchezParejaFlores2001}, the formulation presented here seems to be new. We include the argument above because it is entirely self-contained and avoids any use of linear algebra.
\end{remark}

\end{document}